\documentclass[12pt]{amsart}
\usepackage{amsmath}
\usepackage{amsxtra}
\usepackage{amscd}
\usepackage{amsthm}
\usepackage{amsfonts}
\usepackage{amssymb}
\usepackage{mathrsfs}
\usepackage[all]{xy}
\usepackage[pdftex]{color,graphicx}
\usepackage{graphicx}
\usepackage{color}
\usepackage{mathrsfs}

\baselineskip 18pt \textwidth 16cm  \theoremstyle{plain}

\newtheorem{theorem}{Theorem}[section]
\newtheorem{proposition}[theorem]{Proposition}
\newtheorem{corollary}[theorem]{Corollary}
\newtheorem{lemma}[theorem]{Lemma}

\theoremstyle{definition}
\newtheorem{definition}[theorem]{Definition}
\newtheorem{example}[theorem]{Example}

\theoremstyle{remark}
\newtheorem{remark}[theorem]{Remark}

\def\Z{\Bbb Z}

\def\N{\Bbb N}

\def\e{\varepsilon}

\def\nom{\triangleleft}
\def\la{\langle}
\def\ra{\rangle}

\DeclareMathOperator\GL{GL} 
 
 \DeclareMathOperator\SL{SL}
 
 \DeclareMathOperator\Mat{M}

\DeclareMathOperator\diag{diag}
\DeclareMathOperator{\iden}{Id}
\DeclareMathOperator{\tr}{tr}

\DeclareMathOperator{\val}{val}

\makeatletter

\newcommand{\Rmnum}[1]{\expandafter\@slowromancap\romannumeral #1@}
\makeatother

\title{Words have bounded width in $\SL(n,\Z)$}

\author{Nir Avni}

\address{Nir Avni,
Department of Mathematics, Northwestern University, Evanston IL, USA.}
\email{avni.nir@gmail.com}

\author{Chen Meiri}
\address{Chen Meiri,
Department of Mathematics, Technion, Haifa, Israel.}
\email{chenm@technion.ac.il}



\begin{document}

\maketitle

\begin{abstract} We prove two results about width of words in $\SL_n(\mathbb{Z})$. The first is that, for every $n \ge 3$, there is a constant $C(n)$ such that the width of any word in $\SL_n(\mathbb{Z})$ is less than $C(n)$. The second result is that, for any word $w$, if $n$ is big enough, the width of $w$ in $\SL_n(\mathbb{Z})$ is at most 87.

\end{abstract}

\section{Introduction}

A word is an element in a free group. Given a word $w=w(x_1,\ldots,x_d)\in F_d$ and a group $\Gamma$, we have the word map $w:\Gamma ^d \rightarrow \Gamma$ defined by substitution. The set of $w$-values is the set
\[
w(\Gamma) := \left\{ w(g_1,\ldots,g_d), w(g_1,\ldots,g_d) ^{-1} \mid g_i\in \Gamma \right\}.
\]
Sets of word values in many families of groups were extensively studied. See the book \cite{Seg} and the references therein for results on free and hyperbolic groups, nilpotent groups, $p$-adic analytic groups, and general finite groups (the last part is the main ingredient in the proof by Nikolov and Segal of Serre's conjecture that any finite-index subgroup in a finitely generated pro-finite group is open). We briefly describe some of the results that are relevant to this work.

Sets of word values in algebraic groups are large: Borel proved in \cite{Bo} that if $w$ is a non-trivial word and $G$ is a connected simple algebraic group defined over an algebraically closed field $k$, then $w(G(k))$ contains a Zariski-open dense set. For Lie groups, the situation is more complicated. For example, Thom \cite[Corollary 1.2]{Th} and Lindenstrauss (unpublished) proved that  sets of word values in the unitary group $U_n$ can have arbitrarily small radii. Nevertheless, Borel's theorem implies that, for any semisimple Lie group $\mathbf{G}$ and any non-trivial word $w$, the set of word values $w(\mathbf{G})$ contains an open ball. It follows that, if $\mathbf{G}$ is compact, there is a constant $C$ (depending on $\mathbf{G}$ and $w$) such that any element of $\mathbf{G}$ is a product of at most $C$ word values. For arithmetic groups, sets of word values are very mysterious, even for simple words. For example, for every 
$n \ge 3$, the question whether every element of $\SL_n(\mathbb{Z})$ is a commutator is widely open. We do, however, know that the set of commutators in $\SL_n(\mathbb{Z})$ is quite large: Dennis and Vasserstein proved in \cite{DV} that every element in $\SL_n(\mathbb{Z})$ is a product of at most six commutators if $n$ is large enough. 

A remarkable theorem of Larsen and Shalev \cite{LS} says that a stronger statement holds for finite simple groups: for every non-trivial word $w$, if $\Gamma$ is a large enough finite simple group, then every element of $\Gamma$ is a product of two word values. 

Our first result generalizes the theorem of Dennis and Vasserstein's in a form similar to the theorem of Larsen and Shalev:

\begin{theorem} \label{thm:intro.fix.w} There is a constant $C$ with the following property: for any word $w$, there is $n_w$ such that, for all $n>n_w$, every element of $\SL_n(\mathbb{Z})$ is a product of at most $C$ elements of $w(\SL_n(\mathbb{Z}))$. In fact, $C$ can be taken to be equal to 87. 
\end{theorem} 

In general, we cannot expect the subgroup generated by $w(\SL_n(\mathbb{Z}))$ to be equal to $\SL_n(\mathbb{Z})$. We define the width of $w$ in $\SL_n(\mathbb{Z})$ to be the minimum of the numbers $C$ such that any element of $\langle w(\SL_n(\mathbb{Z})) \rangle$ is a product of at most $C$ elements of $w(\SL_n(\mathbb{Z}))$. If no such number exists, we say that the width of $w$ is infinite.

Our next theorem provides uniform bounds for width, for a fixed $n$:

\begin{theorem}\label{thm:intro.fix.n} For any $n \geq 3$ there is an integer $C=C(n)$ such that, for any word $w$, the width of $w$ in $\SL_n(\mathbb{Z})$ is less than $C$.
\end{theorem} 

\begin{remark} Theorems \ref{thm:intro.fix.w} and \ref{thm:intro.fix.n} are optimal in the following sense: for every $C$ there are infinitely many pairs $(n,w)$ such that the width of $w$ in $\SL_n(\mathbb{Z})$ is greater than $C$. This easily follows from \cite[Theorem 1]{Lub}.
\end{remark} 

We do, however, have the following result which is uniform in $n$ and $w$:

\begin{theorem} \label{thm:intro.congruence.subgroup} There is a constant $C$ such that, for every any non-trivial word $w$, there is $d=d(w)\in \mathbb{Z}$ such that for every $n \ge 3$, every element of the $d$-congruence subgroup $\SL_n(\mathbb{Z};d)$ is a product of at most $C$ elements of $w(\SL_n(\mathbb{Z}))$. If $n$ is large enough, $C$ can be taken to be equal to 80.
\end{theorem} 

\begin{remark} Let $O$ be the ring of integers in a number field, let $S$ be a finite set of primes of $O$, and let $O_S$ denote the localization of $O$ by $S$. The proofs below also show similar bounds for $\SL_n(O_S)$, but the bounds obtained by these proofs depend on $O_S$. While we do not know whether widths of words in $\SL_n(O_S)$ are bounded uniformly in $O_S$, Corollary 4.6 of \cite{MRS} gives some indication that this is indeed the case. 
In another direction, we do not even know whether words in other higher-rank non-uniform lattices (especially non-split) have finite width. We exclude lattices of rank one from the discussion since these include free groups and hyperbolic groups for which the the width of every non-trivial word is infinite, see \cite{MN}.
\end{remark} 

\begin{remark} Let $\Gamma$ be an irreducible arithmetic lattice in a higher-rank semisimple group $G$, and assume that there is a compact semisimple Lie group $K$ and a dense embedding $\pi : \Gamma \hookrightarrow K$ (this implies that $\Gamma$ is cocompcat in $G$). By the result of Thomas and Lindenstrauss mentioned above, there are words $w\in F_2$ such that $\pi (w(\Gamma))$ is contained in an arbitrarily small neighborhood of the identity. It follows that the width of $w$ can be arbitrarily large. This means that the analog of Theorem \ref{thm:intro.fix.n} fails for $\Gamma$. Noting that the image under $\pi$ of any finite-index subgroup of $\Gamma$ is dense, we get that Theorem \ref{thm:intro.congruence.subgroup} also fails. We do not know whether every word has finite width in higher-rank cocompact lattices, nor whether the analog of Theorems \ref{thm:intro.fix.w} holds for the class of cocompact lattices.
\end{remark} 

We briefly sketch the proofs of the main theorems. For $n \geq 2$ and $q\in \mathbb{Z}$, denote by $U_n(\mathbb{Z};q)$ the subgroup of all unipotent upper triangular matrices in $\SL_n(\mathbb{Z})$ whose off-diagonal entries are divisible by $q$. Denote similarly $L_n(\mathbb{Z};q)$, replacing upper triangular by lower triangular. Finally, for a group $G$, a subset $X \subset G$, and a natural number $n$, we denote $X^n=\left\{ x_1 \cdots x_n \mid x_i\in X\cup \left\{ 1 \right\} \right\}$.

The main step is to prove the following theorem:

\begin{theorem} \label{thm:intro.bounded.generation} There is a constant $C$ such that, for any $n \geq 3$ and any $q\in \mathbb{Z}$, $(U_n(\mathbb{Z};q)L_n(\mathbb{Z};q))^C$ is a finite-index subgroup of $\SL_n(\mathbb{Z})$.
\end{theorem} 

Theorem \ref{thm:intro.bounded.generation} is proved by induction on $n$ in \S\ref{sec:ntro.bounded.generation}. The case $n=3$  is essentially due to Carter, Keller, and Paige (see \cite{WM} for an exposition of the proof). The
argument for  induction step follows Dennis and Vaserstein \cite{DV}.

Given Theorem \ref{thm:intro.bounded.generation}, we deduce Theorem \ref{thm:intro.congruence.subgroup} without the explicit bound on $C$ in \S\ref{sec:congruence}. A short arjument implies that $w(\SL_3(\mathbb{Z}))^2$ contains an elementary matrix. Using various embeddings of $\SL_3(\mathbb{Z})$ into $\SL_n(\mathbb{Z})$, we show that $w(\SL_n(\mathbb{Z})^{C'}$ contains $U_n(\mathbb{Z};q)$ and $L_n(\mathbb{Z};q)$ for some $C$ and $q$. Theorems \ref{thm:intro.fix.w} and \ref{thm:intro.fix.n} follow from Theorem \ref{thm:intro.congruence.subgroup}, a $p$-adic open mapping theorem, and the Larsen--Shalev theorem \cite{LS}.

\subsection{Acknowledgements} The authors thank Andrei Rapinchuk and Ariel Karelin for helpful conversations.
They are also thankful to the anonymous referee for improving the bounds of Theorems \ref{thm:intro.fix.w} and \ref{thm:intro.congruence.subgroup} and for simplifying the proof of Lemma \ref{lem:ad.irr}. 
N.A. was partially supported by NSF grant DMS-1303205 and BSF grant 2012247. C.M. was partially supported by BSF grant 2014099 and ISF grant 662/15.

\section{Proof of Theorem \ref{thm:intro.bounded.generation}}\label{sec:ntro.bounded.generation}

In this section, we prove Theorem \ref{thm:intro.bounded.generation}. We start by setting up the notations and recalling some facts. 

\begin{definition} Let $A$ be a commutative ring with unit, let $I$ be an ideal in $A$, and let $n \geq 2$ be an integer. \begin{enumerate}
\item $\SL_n(A;I)$ is the subgroup of $\SL_n(A)$ consisting of the matrices which are congruent to the identity matrix  modulo the ideal $I$. The subgroup $\SL_n(A;I)$ is called the $I$-congruence subgroup of $\SL_n(A)$.
\item $U_n(A;I)$ is the subgroup of $\SL_n(A;I)$ consisting of unipotent upper triangular matrices.
\item $L_n(A;I)$ is the subgroup of $\SL_n(A;I)$ consisting of unipotent lower triangular matrices.
\item In the case where $A=\Z$ and $I=q\Z$ we sometimes write
$\SL_n(\Z;q)$, $U_n(\Z;q)$ and $L_n(\Z;q)$ instead of $\SL_n(\Z;I)$, $U_n(\Z;I)$ and $L_n(\Z;I)$.

\end{enumerate} 
\end{definition} 

\begin{definition} Let $A$ be a commutative ring with a unit, let $I$ be an ideal in $A$, and let $n \geq 2$ be an integer. \begin{enumerate}
\item For $x\in A$ and $1 \leq i \neq j \leq n$, let $e_{i,j}(x)$ denote the $n$-by-$n$ matrix with 1s along the diagonal, $x$ in the $(i,j)$-entry, and zero in all other entries.
\item Denote by $E(n,A;I)$ the subgroup generated by the elementary matrices $e_{i,j}(x)$, for $x\in I$. We will write $E(n,A)$ instead of $E(n,A;A)$.
\item Denote by $E^\triangleleft(n,A;I)$ the normal subgroup of $E(n,A)$ generated by $E(n,A;I)$.
\item In the case where $A=\Z$ and $I=q\Z$ we sometimes write
$E(n,\Z;q)$ and $E^\triangleleft(n,\Z;q)$ instead of $E(n,\Z;I)$  and $E^\triangleleft(n,\Z;I)$.
\end{enumerate} 
\end{definition}

The following is \cite[Proposition 2]{Tits}:

\begin{proposition}[Tits] \label{prop:Tits} If $A$ is a commutative ring, $I$ is an ideal of $A$, and $n \geq 3$, then $E^{\triangleleft}(n,A;I^2) \subseteq \langle U_n(A ; I) \cup L_n(A ; I)\rangle$.
\end{proposition}

The following is proved in \cite{WM}:
\begin{theorem}[Carter-Keller-Paige] \label{prop:CKP}\label{prop:CKP} There is a first-order statement $\varphi$ in the language of rings with the following properties: 
\begin{enumerate} 
\item $\varphi$ holds in $\mathbb{Z}$.
\item \label{it:bdd.gen.E} If $A$ is a ring satisfying $\varphi$ and $I$ is an ideal of $A$, then $[\SL_n(A;I):E^{\triangleleft}(n,A;I)] \leq 2 \cdot 8!$.
\end{enumerate} 
\end{theorem}

\begin{remark} Theorem \ref{prop:CKP} is proved in \cite{WM}. More precisely, if we take $\varphi$ to be the conjunction of the conditions $SR_{1\frac12}$, $Gen(2 \cdot 8!,1)$, $Exp(2 \cdot 8!,2)$ (see \cite[Definitions 2.10, 3.2, 3.6]{WM}), then \cite[Lemma 2.13, Corollary 3.5, Theorem 3.9]{WM} imply that $\mathbb{Z}$ satisfies $\varphi$ and \eqref{it:bdd.gen.E} is \cite[Theorem 3.12]{WM}.
\end{remark}

\begin{corollary}\label{claim:bdd.gen.2} There is a constant $C=C(n)$ such that the following holds: For any $q \in \N^+$, there are $g_1,\ldots,g_{2 \cdot 8!}\in \SL_n(\Z;q^2)$ such that $\SL_n(A;q^2)$ is contained in the union of the translations by $g_1,\ldots,g_{2 \cdot 8!}$ of the set $\left( U_n(\Z;q) L_n(\Z;q) \right)^{C}$.
\end{corollary}
\begin{proof} 
Let $A$ is a ring which is elementarily equivalent to $\mathbb{Z}$ (i.e., satisfies the same first order sentences as $\Z$) and let $I$ be an ideal of $A$. Proposition \ref{prop:Tits} and Theorem \ref{prop:CKP} imply that 
\begin{equation}\label{eq:ele}
\left[ \SL_n(A;I^2) : \SL_n(A;I^2) \cap \langle U_n(A;I) L_n(A;I) \rangle \right] \leq 2 \cdot 8!.
\end{equation}
Assume the Corollary is false. Then, for every $k \in \N$, there are $q_k \in \N^+$ and matrices $g_{k,1},\ldots,g_{k,2 \cdot 8!+1} \in \SL_n(\Z ; q_k^2)$ such that $(g_{k,i}) ^{-1} g_{k,j} \notin \left( U_n(\Z ; q_k) L_n(\Z ; q_k) \right) ^k$ if $i\neq j$. 

Choose a nonprincipal ultrafilter $\mathcal{U}$ on $\mathbb{N}$, and let $A$ be the ultrapower of $\Z$ over $\mathcal{U}$. Then $A$ is elementarily equivalent to $\mathbb{Z}$ and  $\SL_n(A)$ is isomorphic to the ultrapower of $\SL_n(\Z)$ over $\mathcal{U}$. Let $I$ be the ideal of $A$ represented by $\prod_k q_k\Z $, and for every $1 \le i \le k$, let $g_i\in \SL_n(A;I^2)$ be the element represented by $(g_{k,i})_k$. Then $g_1,\ldots,g_{2 \cdot 8!+1}$ belong to different cosets of $\langle U_n(A;I) L_n(A;I) \rangle$, contradicting Equation \eqref{eq:ele}. 
\end{proof}

The following two technical lemmas will be needed in the proof of Proposition \ref{claim:bdd.gen.3} below. 

\begin{lemma} \label{lem:double} Let $G$ be a group, let $X \subset G$ be a symmetric set such that there are $d$ translates of $X$ that cover $G$. Then $X^{4d+2}$ is a group.
\end{lemma} 

\begin{proof} Denote $Y=X^2$. Then $1\in Y$ and there are $d$ translates of $Y$ that cover $G$. 
Since $1 \in Y$, $Y^k \subseteq Y^{k+1}$ for every $k$. 
It is enough to show that $Y^{k}=Y^{k+1}$ for some $k \leq 2d+1$. Suppose that $G=\cup_{i=1}^d g_iY$
for some $g_1,\ldots,g_d \in G$. We can assume that $g_1=1$. For every $k$, if $Y^{k} \neq Y^{k+1}$, choose $h\in Y^{k+1} \smallsetminus Y^k$. By assumption, there is $i$ such that $h\in g_iY$. Then $g_i\in Y^{k+2}$ but $g_i \notin Y^k$. By induction we see that if $Y^{2k-1} \neq Y^{2k}$ for some $1 \le k$, then 
$Y^{2k+1}$ contains at least $k$ distinct $g_i$s. This implies that $Y^{2d+1}=Y^{2d+2}$.  
\end{proof}

\begin{lemma}\label{lem:generation} Let $K \subseteq H \subseteq G$ be groups
such that $[H:K] < \infty$. Let
$X \subseteq G$ be a symmetric subset. Assume that $HX=G$ and that $K \subseteq X$. Then, $X^{4[H:K]}$ is a subgroup.
\end{lemma}

\begin{proof} Since $1 \in K \subseteq X$, the sets $\left( X^n K \cap H \right) \subseteq H$ are non-decreasing. Hence, there is $n \leq 4[H:K]-3$ such that 
$$X^{n} K \cap H=X^{n+1} K \cap H=X^{n+2} K \cap H=X^{n+3} K \cap H=X^{n+4} K \cap H .$$ Since $HX=G$, we have $X^{n+3} \subseteq \left( X^{n+4} \cap H \right) X $. Thus,
\[
X^{n+3} \subseteq \left( X^{n+4} \cap H \right) X \subseteq \left( X^{n+4} K \cap H \right) X \subseteq \left( X^{n} K\cap H \right) X \subseteq X^{n+2},
\]
so $X^{n+2}$ is a group.
\end{proof}

\begin{proposition}\label{claim:bdd.gen.3} There is a constant $D=D(n)$ such that, for any $q \in \N^+$, the set $\left( U_n(\Z;q) L_n(\Z;q) \right)^{D}$ is a group, and, therefore, equals to $\langle U_n(\Z;q) L_n(\Z;q) \rangle$.
\end{proposition}
\begin{proof}
For any $k$, the set $(L_n(\Z;q) U_n(\Z;q))^{k+1}$ contains the symmetric subset $(L_n(\Z;q) U_n(\Z;q))^k \cup (U_n(\Z;q) L_n(\Z;q))^k$. Corollary \ref{claim:bdd.gen.2} and Lemma \ref{lem:double} imply that there is a constant $D'$ such that $(L_n(\Z;q) U_n(\Z;q))^{D'}$ contains a subgroup $S(I)$ of $\SL_n(\Z;q^2)$ of index at most $2 \cdot 8!$.

Note that $\SL_n(\Z,q)/\SL_n(\Z,q^2)$ is abelian so $\SL_n(\Z,q^2)L_n(\Z;q) U_n(\Z;q)$ is a subgroup of $\SL_n(\Z)$. The desired result follows by applying Lemma \ref{lem:generation} to $K=S(I)$, $H=\SL_n(\Z;q^2)$, $G=\SL_n(\Z,q^2)L_n(\Z;q) U_n(\Z;q)$, and $X=(L_n(\Z;q) U_n(\Z;q))^{D'} \cup (U_n(\Z;q) L_n(\Z:q))^{D'} \subseteq (L_n(\Z;q) U_n(\Z;q))^{D'+1}$:
\end{proof}

In order to prove Theorem \ref{thm:intro.bounded.generation} we have to show that the constant 
$D(n)$ in Proposition \ref{claim:bdd.gen.3} can be made independent of $n$. The following technical generalization of Proposition \ref{prop:Tits} is needed. 

\begin{lemma}\label{lemma: Gen.Tits} Let $n \ge 3$ and let $I$ be an ideal in a commutative ring $A$. Then $E^\nom(n+1; A ,I^2)$ is contained in the subgroup $$K(I):=  \la e_{i,j}(a) \mid 1 \le i \neq j \le n+1 , \{i,j\}\neq \{1,n+1\} , a \in I \ra.$$
\end{lemma}

\begin{proof} 
We follow the proof of Proposition \ref{prop:Tits} in \cite{Tits}.
 
Let $1 \le i \ne j \le n+1$, $1 \le r \ne s \le n+1$ and $a,b\in A$. Recall the following relations: \begin{equation}\label{eq:rel} 
\begin{cases}
e_{r,s}(b)e_{i,j}(a)e_{r,s}(b)^{-1}=e_{i,j}(a)e_{i,s}(-ab) & \text{ if } j = r \text{ and } i \ne s \\
e_{r,s}(b)e_{i,j}(a)e_{r,s}(b)^{-1}=e_{i,j}(a)e_{r,j}(ab) & \text{ if } j \ne r \text{ and } i = s \\
e_{r,s}(b)e_{i,j}(a)e_{r,s}(b)^{-1}=e_{i,j}(a) & \text{ if } j \ne r \text{ and } i \ne s
\end{cases}.
\end{equation}

For every $1 \le i \ne j \le n+1$, denote $F_{i,j}(I^2):=\la e_{i,j}(a),e_{j,i}(a) \mid a \in I^2 \ra $. Let  $F_{i,j}^\nom(I^2)$ be the minimal normal subgroup of 
$F_{i,j}:= F_{i,j}(A)$ which contains $F_{i,j}(I^2)$. Define $F^\nom (I^2):=\la F^\nom_{i,j}(I^2) \mid 1 \le i \ne j \le n+1\ra$. Equation \eqref{eq:rel} implies that 
for every $1 \le i \neq j \le n+1$ and every $a \in A$, $e_{i,j}(a)F^\nom(I^2)e_{i,j}(a)^{-1}=F^\nom(I^2)$. 
Thus $F^\nom(I^2)$ is a normal subgroup of $E(n+1,A)$ containing all $e_{i,j}(a)$, $a\in I^2$, so it must be equal to $E^\nom(n+1,A,I^2)$.  Thus, in order to finish the proof it is enough to show that for every $1 \le i < j \le n+1$,  $F^\nom_{i,j}(I^2) \subseteq  K(I)$.

Let $E^+(n,A;I)$ and $E^-(n,A;I)$ be the images of $E(n,A,I)$ in $\SL_{n+1}(A)$ under the embeddings 
$M \mapsto \left(\begin{matrix} M &0 \\ 0 & 1 \end{matrix}\right)$ and $M \mapsto \left(\begin{matrix} 1 &0 \\ 0 & M \end{matrix}\right)$. 
By applying Proposition \ref{prop:Tits} with respect to $E^+(n,A;I)$ and $E^-(n,A;I)$, we see that $K(I)$ contains $F^\nom_{i,j}(I^2)$ for every $1 \le i  <  j \le n+1$ such that $(i,j)\ne (1,n+1)$. 
 
 Equation \eqref{eq:rel} implies that $K(I)$ in normalized by $e_{1,n+1}(a)$ and $e_{n+1,1}(a)$ for every $a \in R$. 
For every $a,b \in I$, $e_{1,n+1}(ab)=[e_{1,2}(a),e_{2,n+1}(-b)]\in K(I)$ and $e_{n+1,1}(ab)=[e_{n+1,2}(a),e_{2,1}(-b)]\in K(I)$.
Thus, $F^\nom_{1,n+1}(I^2) \le K(I)$.
\end{proof}

The next Lemma is the key ingredient in the proof of Theorem \ref{thm:intro.bounded.generation}.

\begin{lemma}\label{lem:up} 
Let $n \ge 3$ and $q \in \N^+$ and assume that 
$\left( U_n(\mathbb{Z};q) ) L_n(\mathbb{Z};q) ) \right)^{D}= \la U_n(\mathbb{Z};q) L_n(\mathbb{Z};q)  \ra $. Then for every $m \ge n$, 
$\left( U_m(\mathbb{Z};q) ) L_m(\mathbb{Z};q) ) \right)^{D}= \la U_m(\mathbb{Z};q) L_m(\mathbb{Z};q)  \ra $. 
\end{lemma}

\begin{proof}
The proof follows the proof of Lemma 7 of \cite{DV} and is by induction on $m$. The base case $m=n$ is clear. It remains to show that if the claim is true for some $m \ge 3$ then it is also true for $m+1$. 

Let $T:=\left( U_{m+1}(\Z ;q) L_{m+1}(\Z ; q) \right) ^D$ and $H=\left\{ g\in \SL_{m+1}(\Z) \mid gT=T \right\}$. Since $H$ is a group, it is enough to prove that $H$ contains both $U_{m+1}(\Z ; q)$ (which is clear) and $L_{m+1}(\Z ; q)$.

We embed $L_m(\Z;q)$ and $U_m(\Z ;q)$ in $\SL_{m+1}(\Z; q)$ by the embedding $M \mapsto \left(\begin{matrix} M &0 \\ 0 & 1 \end{matrix}\right)$. 
We denote the abelian group  $\la e_{i,m+1}(a) \mid 1 \le i \le m,\ a \in q\Z \ra$ by $C_{m+1}(\Z; q)$ and the abelian group $\la e_{m+1,i}(q) \mid 1\le i \le m,\ a \in q\Z \ra$ by $R_{m+1}(\Z; q)$. We have that $U_{m+1}(\Z; q)=U_m(\Z; q) \ltimes C_m(\Z; q)$, $L_{m+1}(\Z; q)=L_m(\Z; q) \ltimes R_m(\Z; q)$, and that $U_m(\Z; q)$ and $L_m(\Z; q)$ each normalizes both $C_m(\Z; q)$ and $R_m(\Z; q)$.  The induction hypothesis implies that:
\[
L_m(\Z: q) \left( U_{m+1}(\Z: q)L_{m+1}(\Z; q) \right) ^D =
\]
\[
= L_m(\Z: q) \left( U_m(\Z: q)  C_m(\Z: q)  L_m(\Z: q)  R_m(\Z: q) \right) ^D =
\]
\[
= L_m(\Z: q) \left( U_m(\Z: q)  L_m(\Z: q) \right) ^D \cdot  \left(C_m(\Z: q)R_m(\Z: q)\right)^D = 
\]
\[
=\left( U_m(\Z: q)  L_m(\Z: q) \right) ^D \cdot \left(C_m(\Z: q) R_m(\Z: q)\right)^D= 
\]
\[
=\left( U_m(\Z: q)  C_m(\Z: q)  L_m(\Z: q)  R_m(\Z: q) \right) ^D =
\]
\[
=\left( U_{m+1}(\Z: q)L_{m+1}(\Z: q) \right) ^D.
\]
Hence, $L_m(\Z: q) \subseteq H$, i.e., for every $1 \leq i < j \leq m$ and $a \in I$, we have $e_{j,i}(a) \in H$. Arguing similarly using the embedding $M \mapsto \left(\begin{matrix} 1 &0 \\ 0 & M \end{matrix}\right)$, we get that $e_{j,i}(a)\in H$, for every $2 \leq i < j \leq m+1$ and $a \in I$. 
It remains to show that for every $a \in I$, $e_{m+1,1}(a)\in H$.

The main theorem of \cite{Men} says that $E^\nom(n,\Z,k)=\SL_n(\Z,k)$ for every  $k \in \N^+$. 
Thus, Lemma \ref{lemma: Gen.Tits}  implies that $\SL_{m+1}(\Z;q^2)=E^\nom(m+1,\Z; q^2) \subseteq H$. Since $\SL_{m+1}(\Z;q)/\SL_{m+1}(\Z;q^2)$ is abelian, $e_{m+1,1}(a)U_{m+1}(\Z: q) e_{m+1,1}(a) ^{-1} \subseteq \SL_{m+1}(\Z; q^2) \cdot U_{m+1}(\Z: q)$, for every $a \in I$. It follows that, for every $a \in I$, 
\[
e_{m+1,1}(a) \left( U_{m+1}(\Z: q)L_{m+1}(\Z: q) \right) ^D =
\]
\[
= e_{m+1,1}(a) U_{m+1}(\Z: q) e_{m+1,1}(a) ^{-1} e_{m+1,1}(a) L_{m+1}(\Z: q) \left( U_{m+1}(\Z: q)L_{m+1}(\Z: q) \right) ^{D-1} \subseteq 
\]
\[
\subseteq \SL_{m+1}(\Z; q^2) \cdot U_{m+1}(\Z: q) L_{m+1}(\Z: q) \left( U_{m+1}(\Z: q)L_{m+1}(\Z: q) \right) ^{D-1} =
\]
\[
=\left( U_{m+1}(\Z: q)L_{m+1}(\Z: q) \right) ^D.
\]
In particular, for every $a \in I$, $e_{m+1,1}(a)\in H$. Hence, $L_{m+1}(\Z: q) \subseteq H$ as desired. \end{proof}

\begin{proof}[Proof of Theorem \ref{thm:intro.bounded.generation}]

Proposition \ref{claim:bdd.gen.3}
implies that there is a constant $C=D(3)$ such that $\left( U_3(\mathbb{Z};q) ) L_3(\mathbb{Z};q) ) \right)^{C}= \la U_3(\mathbb{Z};q) L_3(\mathbb{Z};q)  \ra $. Lemma \ref{lem:up} implies that 
$\left( U_n(\mathbb{Z};q) ) L_n(\mathbb{Z};q) ) \right)^{C}= \la U_n(\mathbb{Z};q) L_n(\mathbb{Z};q)  \ra $ for every $n \ge 3$. Proposition \ref{prop:Tits} implies that 
$\left( U_n(\mathbb{Z};q) ) L_n(\mathbb{Z};q) ) \right)^{C}$ contains the congruence subgroup $\SL_n(\Z;q^2)$ and this subgroup has a finite index in $\SL(n,\Z)$. 
\end{proof}

\section{Proof of Theorem \ref{thm:intro.congruence.subgroup} without an explicit bound} \label{sec:congruence} 

We will need the following lemma, which we state without a proof:
\begin{lemma} \label{lem:q.above.diagonal} All upper-triangular matrices $g\in U_n(\mathbb{Z} ; q)$ such that $g_{i,i+1}=q$, for all $i$, are conjugate.
\end{lemma} 

\begin{proof}[Proof of Theorem \ref{thm:intro.congruence.subgroup} without an explicit bound] 
Identify $\SL_2(\Z)$ with its image in $\SL_3(\Z)$ under the embedding $M \mapsto \left(\begin{matrix} M &0 \\ 0 & 1 \end{matrix}\right)$. Since $\SL_2(\Z)$ contains a non-abelian free group there exists $\pm I_2 \ne g \in w(\SL_2(\Z))$. 
There exists $h \in \la e_{1,3}(1),e_{2,3}(1)\ra$ such that $[g,h]=g^{-1}h^{-1}gh$ is a non-trivial element and this element is
conjugate to $e_{1,3}(q)$ for some positive $q \in \N$. For the chosen $g$ and $h$, we have $[g,h]^n=[g,h^n] \in w(\SL_3(\Z))^2$. Since $w(\SL_3(\Z))^2$ is a normal subset, $\la e_{1,3}(q)\ra \subseteq w(\SL_3(\Z))^2$. We will show that the statement of Theorem \ref{thm:intro.congruence.subgroup} holds with respect to $d=q^2$.

We claim that  for any integers $a_1,\ldots,a_{n-1}$, there is $g\in w \left( \SL_n(\mathbb{Z}) \right)^{8} \cap U_n(\mathbb{Z} ; q)$ such that 
for every $i$, $g_{i,i+1}=qa_i$.
Using two different embeddings of the group $\SL_3(\mathbb{Z}) \times \cdots \times \SL_3(\mathbb{Z})$ ($\lfloor n/3 \rfloor$ times) into $\SL_n(\mathbb{Z})$ as block-diagonal matrices, we get that there is a matrix $g^1\in w\left( \SL_n(\mathbb{Z}) \right)^{4} \cap U_n(\mathbb{Z} ; q)$ such that $g^1_{i,i+1}=q a_i$ if $i \equiv 1 \text{ (mod 3)}$ and $g^1_{i,i+1}=0$ otherwise. Using one  embedding of the group $\SL_3(\mathbb{Z}) \times \cdots \times \SL_3(\mathbb{Z})$ ($\lfloor n/3 \rfloor$ times) into $\SL_n(\mathbb{Z})$ as block-diagonal matrices, we get that there is a matrix $g^2\in w\left( \SL_n(\mathbb{Z}) \right)^{2} \cap U_n(\mathbb{Z} ; q)$ such that $g^2_{i,i+1}=q a_i$ if $i \equiv 2 \text{ (mod 3)}$ and $g^2_{i,i+1}=0$ otherwise. Similarly, there is $g^3 \in w\left( \SL_n(\mathbb{Z}) \right)^{2} \cap U_n(\mathbb{Z} ; q)$  such that $g^3_{i,i+1}=q a_i$ if $i \equiv 3 \text{ (mod 3)}$ and $g^3_{i,i+1}=0$ otherwise. The matrix $g=g^0g^1g^2\in w\left( \SL_n(\mathbb{Z}) \right)^{8} \cap U_n(\mathbb{Z} ; q)$ satisfies $g_{i,i+1}=q a_i$. The proof of the claim in now complete. 

It follows from Lemma \ref{lem:q.above.diagonal} that $w(\SL_n(\mathbb{Z}))^{8}$ contains all elements $g\in U_n(\mathbb{Z};q)$ such that $g_{i,i+1}=q$ for every $i$.

Next, we claim that $U_n(\mathbb{Z} ; q) \subseteq w\left( \SL_n(\mathbb{Z}) \right)^{16}$. Indeed, let $h\in U_n(\mathbb{Z} ; q)$.  
There is an element $f\in w\left( \SL_n(\mathbb{Z}) \right)^{8} \cap  U_n(\mathbb{Z} ; q)$ such that
for every $i$,  $f_{i,i+1}=q-h_{i,i+1}$. Then $hf\in U_n(\mathbb{Z};q)$ and for every $i$, $(fh)_{i,i+1}=q$, so $fh\in (w \left( \SL_n(\mathbb{Z}) \right) ^{8}$. Since 
$ w\left( \SL_n(\mathbb{Z}) \right)^{8}$ is symmetric,
it follows that $h\in w\left( \SL_n(\mathbb{Z}) \right)^{16}$. Similarly, $L_n(\mathbb{Z} ; q) \subseteq w\left( \SL_n(\mathbb{Z}) \right)^{16}$.

By Theorem \ref{thm:intro.bounded.generation}, there is a constant $C$ (independent of $q$) such that $$\la U_n(\Z;q)U_n(\Z;q)  \ra =(U_n(\Z;q)U_n(\Z;q))^C \subseteq w\left( \SL_n(\mathbb{Z}) \right)^{32C}.$$ Propositon \ref{prop:Tits} implies that $\SL_n(\Z,q^2) \le \la U_n(\Z;q)U_n(\Z;q)  \ra$.
\end{proof} 

\section{Proof of Theorems \ref{thm:intro.fix.n} and \ref{thm:intro.fix.w}}

In order to deduce Theorems \ref{thm:intro.fix.n} and \ref{thm:intro.fix.w} from Theorem \ref{thm:intro.congruence.subgroup}, we need to study word values in $\SL_n(\mathbb{Z} / q\Z)$ uniformly in $q$. Equivalently, we need to study word values in $\SL_n ( \widehat{\mathbb{Z}})$ where $\widehat{\mathbb{Z}}$ is the profinite completion of $\Z$. We will use a version of the open mapping theorem which is well known, but for which we were unable to find a reference.

For $a\in \mathbb{Z}_p^n$, denote $\| a\|=\max \left\{ |a_i|_p \right\}$, where $|a|_p$ is the $p$-adic valuation of $a$. The function $d(a,b)=\|a-b\|$ is a metric on $\mathbb{Z}_p^n$. Let $X \subset \mathbb{A}_{\mathbb{Z}_p} ^n$ be an affine $\mathbb{Z}_p$-scheme, i.e. the zero locus of a collection of polynomials in $\mathbb{Z}_p[x_1,\ldots,x_n]$. We denote the set of solutions of $X$ with coordinates in $\mathbb{Z}_p$ by $X(\mathbb{Z}_p)$. The restriction of $d$ to $X(\mathbb{Z}_p)$ is a metric on $X(\mathbb{Z}_p)$\footnote{This metric is independent of the affine embedding, but we will not use this fact.}. Let $\mathbb{Z}_p[X]$ be the ring of regular functions on $X$ (the restrictions of polynomials with $\mathbb{Z}_p$ coefficients to $X$). For $f\in \mathbb{Z}_p[X]$, we define $\val_p(f)=\max \left\{ k \mid f\in p^k \mathbb{Z}_p[X] \right\}$. More generally, if $f:X \rightarrow Y$ is a map of affine $\mathbb{Z}_p$-schemes, we define $\val_p(f)$ as the minimum of the valuations of its coordinates. Note that if $\val(f) \geq k$, then $d(f(a),f(b)) \leq p^{-k}d(a,b)$, for every $a,b\in X(\mathbb{Z}_p)$.

Recall that $X$ is called smooth at $a\in X(\mathbb{Z}_p)$ if there are $\psi_1,\ldots,\psi_c\in \mathbb{Z}_p[x_1,\ldots,x_n]$ such that $X$ is the common zero locus of $\psi_i$ and the reductions of $\nabla \psi_i(a)$ modulo $p$ are linearly independent. In this case $n-c$ is called the dimension of $X$ at $a$.

\begin{lemma}\label{lem:4.1} Let $X \subseteq \mathbb{A}_{\mathbb{Z}_p}^n$ be a $\mathbb{Z}_p$-scheme and $a\in X(\mathbb{Z}_p)$. Assume that $X$ is smooth in $a$. Then there is a subset $S \subset \left\{ 1,\ldots,n \right\}$ such that the coordinate projection $\pi : \mathbb{Z}_p ^n \rightarrow \mathbb{Z}_p^S$ satisfies: \begin{enumerate}
\item\label{item:smooth.1} The restriction of $\pi$ to $X(\mathbb{Z}_p)\cap B(a,p ^{-1})$ is one-to-one where $B(a,p ^{-1})$ is the  closed ball of radius $p^{-1}$ around $a$. 
\item\label{item:smooth.2} $\pi (T_aX(\mathbb{Z}_p))=\mathbb{Z}_p^{S}$.
\end{enumerate} 
\end{lemma} 

\begin{proof} Let $\psi_i$ be as in the definition of smoothness. After permutation of the indices, we can assume that the $c$-by-$c$ matrix $\left( \frac{\partial \psi_i}{\partial x_j}(a) \right)$ is invertible over $\mathbb{Z}_p$. For any $f\in \mathbb{Z}_p[x_1,\ldots,x_n]$ and any $a,b\in \mathbb{Z}_p^n$ with $0<d(a,b)<1$, we have $|f(a)-f(b)-\langle \nabla f(a),a-b \rangle| \leq \| a-b \|^2 < \| a-b \|$. If $a,b\in X(\mathbb{Z}_p)$ and $d(a,b)<1$, we have $\psi_i(a)=\psi_i(b)=0$, so $| \langle \nabla \psi_i(a),a-b \rangle | < \| a-b \|$. If, in addition, $\pi(a)=\pi(b)$, write $a-b=(v,0)$, where $v\in p \mathbb{Z}_p^c$ and then
\[
\left \| \left( \frac{\partial \psi_i}{\partial x_j}(a) \right)v \right\| = \max \left\{ | \langle \nabla \psi_i(a),a-b \rangle | \right\} < \| v \|.
\] 
Since invertible matrices do not decrease norm, this is a contradiction. This complete the proof of \eqref{item:smooth.1}. Denoting $S:=\{c+1,\ldots,n\}$, \eqref{item:smooth.2} readily follows form the assumption  that $\left( \frac{\partial \psi_i}{\partial x_j}(a) \right)$ is invertible.
\end{proof} 

\begin{lemma} \label{lem:open.mapping} Let $X,Y$ be affine $\mathbb{Z}_p$-schemes. Let $f:X \rightarrow Y$ be a morphism, let $a\in X(\mathbb{Z}_p)$, and let $k \geq 0$ be an integer. Suppose that \begin{enumerate}
\item $\val_p(f) \geq k$.
\item $df(a)(T_aX(\mathbb{Z}_p)) \supseteq p^kT_{f(a)}Y(\mathbb{Z}_p)$.
\item $X$ is smooth at $a$ and $Y$ is smooth at $f(a)$. 
\end{enumerate} 
Then $f(X(\mathbb{Z}_p))$ contains the closed ball of radius $p^{-k-1}$ around $f(a)$. 
\end{lemma} 

\begin{proof} We first reduce the claim to the case where $X$ is an affine space. Suppose that $X \subset \mathbb{A} ^n$ is $d$-dimensional. By smoothness, it is given as the zero locus of $\varphi_1,\ldots,\varphi_{n-d}\in \mathbb{Z}_p[x_1,\ldots,x_n]$ such that the reductions modulo $p$ of $\nabla \varphi_i(a)$ are linearly independent. Consider the map $F:\mathbb{A} ^n \rightarrow Y \times \mathbb{A} ^{n-d}$ given by $x \mapsto (f(x),p^k\varphi_1(x),\ldots,p^k\varphi_{n-d}(x))$. Then $F$ satisfies the conditions of the lemma. If the claim holds for $F$, then it holds for $f$.

Next, we reduce the claim to the case where $X$ and $Y$ are affine spaces. Indeed, let $e$ be the dimension of $Y$ at $f(a)$. Item \eqref{item:smooth.1} of Lemma \ref{lem:4.1} allows us to assume that the coordinate projection $\pi: Y \rightarrow \mathbb{A} ^{e}$ is one-to-one on $B(f(a),p ^{-1})$. 
Item \eqref{item:smooth.2} of Lemma \ref{lem:4.1} implies that the function $\pi \circ f$ satisfies the conditions of the lemma, and the claim for $\pi \circ f$ implies the claim for $f$.

Finally, we prove the claim in the case $X=\mathbb{A} ^n$ and $Y=\mathbb{A} ^m$. We can assume that $a=0$ and $f(a)=0$. Since the coefficients of $f$ are in $\mathbb{Z}_p$, we have that $df(a')(\mathbb{Z}_p^n) \supseteq p^k \mathbb{Z}_p^m$, for any $a'\in p\mathbb{Z}_p^n$. Let $b\in p^{k+1}\mathbb{Z}_p^m$. We will construct a sequence $a_\ell \in p \mathbb{Z}_p^n$ such that $\| f(a_\ell)-b\| < p^{-k-\ell}$. Taking a limit point of the $a_\ell$s, we get that $b\in f(\mathbb{Z}_p^n)$.

The sequence $a_\ell$ is defined by recursion starting from $a_0=0$. Given $a_\ell$, the assumptions imply that there is $\epsilon \in p^{\ell+1} \mathbb{Z}_p^n$ such that $df(a_\ell)(\epsilon)=b-f(a_\ell)$. We have
\[
\| f(a_\ell + \epsilon)-b \| = \| f(a_\ell +\epsilon)-f(a_\ell)-df(a_\ell)(\epsilon)+df(a_\ell)(\epsilon)+f(a_\ell)-b\| = 
\]
\[
\| f(a_\ell +\epsilon)-f(a_\ell)-df(a_\ell)(\epsilon) \| \leq p^{-k}\| \epsilon \|^2 < p^{-k-\ell-1},
\]
since the function $x \mapsto f(a_\ell+x)-f(a_\ell)-df(a_\ell)(x)$ is a polynomial without constant or linear term and its coefficients are divisible by $p^k$.
\end{proof} 

%

\begin{definition} For elements $g,h\in \SL_n$, let $\Phi_{g,h}:\SL_n \times \SL_n \rightarrow \SL_n$ be the map $\Phi^R_{g,h}(x,y)=g^x h^y$.
\end{definition} 

\begin{lemma} \label{lem:diff.cc} Let $n \ge 3$ and assume that $a,b\in \SL_n(\mathbb{F}_q)$ generate $\SL_n(\mathbb{F}_p)$ where $\mathbb{F}_q$ is a finite field of order $q$. Then the differential of $\Phi_{a,b}$ at $(1,1)$ is onto.
\end{lemma} 

\begin{proof} After identifying $T_{ab}\SL_n=ab +ab \mathfrak{sl}_n$ and $\mathfrak{sl}_{n}$, the differential of $\Phi_{a,b}$ is $(X,Y) \mapsto (X-X^a)^b +(Y-Y^b)$. Let $\varphi\in \Mat_n(\mathbb{F}_q)^*$ and assume it vanishes on the image of $d \Phi_{a,b}$. Then there is $A\in \Mat_n(\mathbb{F}_q)$ such that $\varphi(X)=\tr(AX)$. For every $Y\in \mathfrak{sl}_n(\mathbb{F}_p)$, $\varphi(Y-Y^b)=0$, so $\tr(Y \cdot (A^{b ^{-1}}-A))=\tr(A(Y-Y^b))=0$.
Thus, $A^{b ^{-1}}-A$ is a scalar. Similarly, using the assumption that $\varphi((X-X^a)^b)=0$, we get that $\left( A^{b ^{-1}} \right)^{a ^{-1}}-A^{b ^{-1}}$ is  a scalar. Using the fact that $A^{b ^{-1} }-A$ is a scalar, we get that $A^{a ^{-1}}-A$ is also a scalar. The set $X=\left\{ g\in \SL_n(\mathbb{F}_q) \mid \text{$A^g-A$ is a scalar} \right\}$ is closed under multiplication. Since $a ^{-1} ,b ^{-1} \in X$, we get $X=\SL_n(\mathbb{F}_q)$. Since $\SL_n(\mathbb{F}_p)$ is perfect and the function $g \mapsto A^g-A$ is a homomorphism between $\SL_n(\mathbb{F}_q)$ and $\mathbb{F}_q\iden$, we get that this homomorphism must be trivial. Hence, $A$ commutes with $\SL_n(\mathbb{F}_q)$, so it must be scalar. It follows that the restriction of $\varphi$ to $\mathfrak{sl}_n(\mathbb{F}_q)$ is zero.
\end{proof} 

\begin{lemma} \label{lem:width.adelic.fix.w} For every non-trivial word $w$, there is $n_0$ such that, for any integer $n \geq n_0$, we have $w(\SL_n(\widehat{\mathbb{Z}}))^7=\SL_n(\widehat{\mathbb{Z}})$. 
\end{lemma} 

\begin{proof} By \cite{LS}, there is $n_0$ such that, if $n \geq n_0$ and $p$ is any prime, then $w(\SL_n(\mathbb{F}_p))^2$ contains all non-scalar matrices. In particular, $w(\SL_n(\mathbb{F}_p))^3=\SL_n(\mathbb{F}_p)$. Fix a prime $p$. Choosing generators $a,b\in \SL_n(\mathbb{F}_p)$ (which are not scalars), there are $g,h\in w(\SL_n(\mathbb{Z}_p))^2$ such that the reductions of $g,h$ modulo $p$ are $a,b$ respectively. We get that $w(\SL_n(\mathbb{Z}_p))^4 \supset \Phi_{g,h}(\SL_n(\mathbb{Z}_p) \times \SL_n(\mathbb{Z}_p))$. It is well known that $\SL_n$ and thus also $\SL_n\times \SL_n$ are smooth at every point. Lemma \ref{lem:diff.cc} and Lemma \ref{lem:open.mapping} imply that  $w(\SL_n(\mathbb{Z}_p))^4$ contains the coset $gh\SL_n(\mathbb{Z}_p;p)$. Hence, $w(\SL_n(\mathbb{Z}_p))^7=\SL_n(\mathbb{Z}_p)$.

Since $w(\SL_n(\widehat{\mathbb{Z}}))=\prod_p w(\SL_n(\mathbb{Z}_p))$, the claim follows.
\end{proof} 

\begin{proof}[Proof of Theorem \ref{thm:intro.fix.w} (without an explicit bound)] By Theorem \ref{thm:intro.congruence.subgroup} and Lemma \ref{lem:width.adelic.fix.w}
\end{proof}

We move on to the proof of Theorem \ref{thm:intro.fix.n}. 

\begin{lemma} \label{lem:ad.irr} For every $n\geq 2$ there is a constant $C$ such that the following holds. If $p$ is a prime and $A\in \mathfrak{sl}_n(\mathbb{F}_p)$ is non-central, then
every element of $\mathfrak{sl}_n(\mathbb{F}_p)$ is equal to the sum of at most $C$ elements of 
 $\left\{ x ^{-1} A x \mid x\in \SL_n(\mathbb{F}_p) \right\}$.
\end{lemma} 

\begin{proof} It is well known that the only non-trivial $\SL_n(\mathbb{F}_p) $-invariant subspace of  $\mathfrak{sl}_n(\mathbb{F}_p)$ is the subset consisting of scalar matrices. Hence, for every $p$, there is a constant $C_p$ such that every element of $\mathfrak{sl}_n(\mathbb{F}_p)$ is equal to the sum of at most $C_p$ elements of  $\left\{ x ^{-1} A x \mid x\in \SL_n(\mathbb{F}_p) \right\}$. Therefore, in order to find a uniform $C$, we can and will assume that $p$ is large. In particular, we assume that $p \ne 2$. 

For $1 \leq i \neq j \leq n$, let $E_{i,j}(a)$ be the matrix whose $(i,j)$th entry is $a$ and is zero otherwise. Note that $E_{1,2}(a)$ is conjugate to $E_{1,2}(ab^2)$, for every $b\in \mathbb{F}_p$. Since every element in $\mathbb{F}_p$ is a sum of two squares, we get that, for any $a\in \mathbb{F}_p^ \times$, any element of the form $E_{1,2}(b)$ is the sum of at most two conjugates of $E_{1,2}(a)$. In particular, there exists a one dimensional linear subspace of 
$\mathfrak{sl}_n(\mathbb{F}_p)$ such that all its elements are sums of two conjugates of  $E_{1,2}(a)$.
Using the fact that the only $\SL_n(\mathbb{F}_p) $-invariant subsapce of  $\mathfrak{sl}_n(\mathbb{F}_p)$ is the subset consisting of scalar matrices once again, we see that, if $a \ne 0$, then every matrix in $\mathfrak{sl}_n(\mathbb{F}_p)$ is the sum of at most $2(n^2-1)$ conjugates of $E_{1,2}(a)$. Therefore, it is enough to prove that there is a constant $C$ such that, for some $a\in \mathbb{F}_p$, the matrix $E_{1,2}(a)$ is a sum of $C$ conjugates of $A$.
We divide the proof into several steps:
\\ \\
{\bf Step A:} Assume that  $A$ is nilpotent. By using Jordan's normal form we see that $A$ is conjugate to a block diagonal matrix
and each block diagonal matrix has the from 
\begin{equation} \label{eq:Jordanl.form}
\left( \begin{matrix}  0 & a &  &  & \\  & 0 & 1 &  &   \\  &  & \ddots & \ddots &  \\  &  &  & 0 & 1 \\  &  &  &  & 0 \end{matrix} \right).
\end{equation}
where $0 \ne a \in \mathbb{F}_p$ (we cannot assume that $a=1$ since we are conjugating with a matrix  in $\SL_n(\mathbb{F}_p)$  and not $\GL_n(\mathbb{F}_p)$).
A straightforward argument implies that it is enough to deal with the case where there is just one block.
Clearly, we can assume that the dimension of this block is at least 3.  
Then, there exists $\e \in \{-1,1\}$ such that the diagonal matrix $\diag(\e,1,-1,\ldots,(-1)^n)$ belongs to $\SL_n(\mathbb{F}_p)$.
Denote $B:=\diag(\e,1,-1,\ldots,(-1)^n)+E_{2,n}(1)$. Then the $n$th coordinate of the first row of $A+B^{-1} AB$ is non-zero while all the other rows equal zero.
Thus, $A+B^{-1} AB$ is conjugate to $E_{1,2}(a)$ for some non-zero $a$. 
\\ \\
{\bf Step B:} Assume that $n=2$. If $A$ is nilpotent, it is conjugate to $E_{1,2}(a)$, for some $a$, and the claim holds. In general, since $A$ is not scalar, it is conjugate to $\left( \begin{matrix} 0 & a \\ b & 0\end{matrix} \right)$, and we can assume that $b\neq 0$. We claim that, if $p \geq 17$, there are $x,y,z\in \mathbb{F}_p^ \times$ such that $x^2+y^2+z^2=0$ and $x^{-2}+y^{-2}+z^{-2} \neq 0$. If this claim holds, then
\[
\left( \begin{matrix} x &  \\  & x ^{-1} \end{matrix} \right)A\left( \begin{matrix} x ^{-1} &  \\  & x \end{matrix} \right) +
\left( \begin{matrix} y &  \\  & y ^{-1} \end{matrix} \right)A\left( \begin{matrix} y ^{-1} &  \\  & y \end{matrix} \right) +
\left( \begin{matrix} z &  \\  & z ^{-1} \end{matrix} \right)A\left( \begin{matrix} z ^{-1} &  \\  & z \end{matrix} \right)=
\]
\[
=\left(\begin{matrix} 0 & 0 \\ b(x^{-2}+y^{-2}+z^{-2}) & 0 \end{matrix}\right),
\]
which is nilpotent.

To prove the claim, let $X$ be the projective curve defined by $x^2+y^2+z^2=0$. Then $X$ has $p+1$ points over $\mathbb{F}_p$, and at most 6 of them have a zero in some coordinate. At most eight of the points of $X$ satisfy the equation $x^{-2}+y^{-2}+z^{-2}=0$ (these all have the form $[x:y:1]$, where $x^4+x^2+1=0$ and $y=\pm x ^{-1}$). In particular, if $p \geq 17$, the claim is true.
\\ \\

{\bf Step C:} Assume $n>2$ and the claim is true for all numbers smaller than $n$. We consider the following cases:

{\bf Case 1C: } Assume that $\det A =0$. By conjugating $A$ we can assume that it is of the form 
\begin{equation} 
\left( \begin{matrix}  0 & * \\  0 & B  \end{matrix} \right)
\end{equation} 
where $B \in \mathfrak{sl}_{n-1}(\mathbb{F}_p)$. 
If $B=0$ then $A$ is a nilpotent matrix and we are done by Step 1. Otherwise, we can assume that $p>n-1$ so $B$ is a non-scalar matrix since its trace is equal to zero. 
Then by the induction hypothesis the sum of a bounded number of conjugates of $A$ is a non-zero nilpotent matrix and we are done by Step 1. 

{\bf Case 2C:} $\det A \neq 0$. By using the rational canonical normal form we see that there exist  non-zero $a,b \in \mathbb{F}_p$ such that 
$A$ is conjugate to a block diagonal matrix and one of the blocks of $A$ (for notational ease, assume it's the first) is of the form: 
\begin{equation} \label{eq:rational.form}
\left( \begin{matrix}  0 & a &  &  & \\  & 0 & 1 &  &   \\  &  & \ddots & \ddots &  \\  &  &  & 0 & 1 \\  b & * & \cdots & * & * \end{matrix} \right).
\end{equation}
Denote $B:=\diag(-1,-1,1,\ldots,1)$. Then, $A+B^{-1}AB$ is a non-zero singular matrix and we are done by case 1C.
\end{proof}

\begin{remark} The proof of Lemma \ref{lem:ad.irr}  can be adapted to work over 
all finite fields of characteristic different than 2. For fields of characteristic 2, the argument of part 3B should be replaced. 
\end{remark}

\begin{lemma} \label{lem:width.adelic.fix.n} For any $n \geq 3$ there is $C$ such that \begin{enumerate} 
\item For any $p$, if $X \subseteq \SL_n(\mathbb{Z}_p)$ is symmetric and invariant to conjugation, then $X^C=\langle X \rangle$.
\item For any non-trivial word $w$, the width of $w$ in $\SL_n(\widehat{\mathbb{Z}})$ is less than $C$.
\end{enumerate} 
\end{lemma} 

\begin{proof} \begin{enumerate} \item Let $k=\min \left\{ i \mid \left( \exists g \in X \right) \text{ $g$ is not central modulo $p^{k+1}$} \right\}$. Clearly, $\langle X \rangle \subseteq Z(\SL_n(\mathbb{Z}_p)) \cdot \SL_n(\mathbb{Z}_p;p^{k})$. We will show that there is $C$ such that $\SL_n(\mathbb{Z}_p;p^{k}) \subseteq X^C$, and it will follow that $X^{C+|Z(\SL_n(\mathbb{Z}_p))|}=\langle X \rangle$, which proves the claim since $|Z(\SL_n(\mathbb{Z}_p))| \leq n$. \begin{itemize}
\item[Case 1:] $k=0$. Let $\overline{X} \subseteq \SL_n(\mathbb{F}_p)$ be the reduction of $X$ modulo $p$. By assumption, $\overline{X}$ is non-central, so \cite[Corollary 1.9]{LS01} implies that there is $C_1$, depending only on $n$, such that $\overline{X}^{C_1}=\SL_n(\mathbb{F}_p)$. Let $a,b\in X^{C_1}$ such that their reductions modulo $p$ generate $\SL_n(\mathbb{F}_p)$. By Lemma \ref{lem:diff.cc}, the differential of the map $\Phi_{\overline{a},\overline{b}}$ at $(1,1)$ is onto, which implies that $d \Phi_{a,b}(\mathfrak{sl}_n^2(\mathbb{Z}_p)) =\mathfrak{sl}_n(\mathbb{Z}_p)$, since $d \Phi_{a,b}$ is $\mathbb{Z}_p$-linear. By Lemma \ref{lem:open.mapping}, $ab\SL_n(\mathbb{Z}_p;p) \subseteq \Phi_{a,b}(\SL_n(\mathbb{Z}_p) \times \SL_n(\mathbb{Z}_p)) \subseteq X^{2C_1}$. It follows that $X^{3C_1}=\SL_n(\mathbb{Z}_p)$.
\item[Case 2:] $k>0$. Let $g\in X$ be such that $g$ is not a scalar modulo $p^{k+1}$. Since $g$ is a scalar modulo $p^{k}$,  there exists $h\in \SL_n(\mathbb{Z}_p)$ such that $g^{|Z(\SL_n(\mathbb{Z}_p))|-1}h ^{-1} g h\in X^{|Z(\SL_n(\mathbb{Z}_p))|}$  belongs to $\SL_n(\mathbb{Z}_p;p^{k})$ and is not a scalar modulo $p^{k+1}$. Since $\SL_n(\mathbb{Z}_p;p^{k})/\SL_n(\mathbb{Z}_p;p^{k+1})=\mathfrak{sl}_n(\mathbb{F}_p)$ as $\SL_n(\mathbb{Z}_p)$-modules, Lemma \ref{lem:ad.irr} implies that there is a constant $C$, independent of $X$, such that $X^C \cdot \SL_n(\mathbb{Z}_p;p^{k+1}) \supseteq \SL_n(\mathbb{Z}_p;p^{k})$. Let $\overline{a}$ be a maximal nilpotent Jordan block and let $\overline{b}=\overline{a}^T$. Note that the intersection of the centralizers of $\overline{a},\overline{b}$ in $\Mat_n$ is the collection of scalar matrices. Choose $a,b\in X^C \cap \SL_n(\mathbb{Z}_p;p^{k})$ whose images in $\SL_n(\mathbb{Z}_p;p^{k}) / \SL_n(\mathbb{Z}_p;p^{k+1})=\mathfrak{sl}_n(\mathbb{F}_p)$ are $\overline{a}$ and $\overline{b}$. We will show that $\Phi_{a,b}:\SL_n \times \SL_n \rightarrow \SL_n$ satisfies the conditions of Lemma \ref{lem:open.mapping}. 

Since $a-1$ is divisible by $p^k$, we have $\val_p(x \mapsto x ^{-1} a x-a) \geq k$. It follows that the derivative of this map also has $p$-valuation at least $k$. Similarly, $\Phi_{a,b}$ satisfies the first condition of Lemma \ref{lem:open.mapping}.

Note that $d \Phi_{a,b}(\mathfrak{sl}(\mathbb{Z}_p)^2) \subset p^k \mathfrak{sl}(\mathbb{Z}_p)$. In order to show the reverse containment, it is enough to show that the composition of $d \Phi_{a,b}$ and the reduction map $p^k \mathfrak{sl}_n(\mathbb{Z}_p) \rightarrow p^k \mathfrak{sl}_n(\mathbb{Z}_p) / p^{k+1} \mathfrak{sl}_n(\mathbb{Z}_p)$ is onto. This composition is the map $(X,Y) \mapsto [\overline{X},\overline{a}]+[\overline{Y},\overline{b}]$ (where $[x,y]$ is the Lie bracket), so we need to show that there is no non-zero linear functional that vanishes on all elements of the form $[\overline{X},\overline{a}]$ and $[\overline{X},\overline{b}]$, for $\overline{X}\in \mathfrak{sl}_n(\mathbb{F}_p)$. Any such functional has the form $\tr(A \cdot)$ for some $A\in\mathfrak{sl}_n(\mathbb{F}_p)$. 
Since $\tr(A[B,C])=\tr([A,B]C)$ for every three matrices $A$, $B$ and $C$, the assumption that $\tr(A[\overline{a},\overline{X}])=0$ for all $\overline{X}\in \mathfrak{sl}_n(\mathbb{F}_p)$ implies that $[A,\overline{a}]=\alpha I$, for some $\alpha$. Similarly, $[A,\overline{b}]=\beta I$, for some $\beta$. From $[A,\overline{a}]=\alpha I$ get (by induction) that $A_{i+1,i}=-i \alpha$, whereas from $[A,\overline{b}]=\beta I$ get that $A_{i+1,i}=A_{i+2,i+1}$. Since $n \geq 3$, we get $\alpha=0$. Similarly, $\beta=0$. Consequently, $A$ commutes with $\overline{a}$ and $\overline{b}$, so $A=0$, a contradiction. 

Applying Lemma \ref{lem:open.mapping} to $\Phi_{a,b}$, we get that any element in $ab\SL_n(\mathbb{Z}_p;p^{k+1})$ is in $\Phi_{a,b}(\SL_n(\mathbb{Z}_p) ^2)$, so, in particular, $ab\SL_n(\mathbb{Z}_p;p^{k+1}) \subset X^{2C}$ and $\SL_n(\mathbb{Z}_p;p^{k+1}) \subset X^{4C}$. Since $X^C \SL_n(\mathbb{Z}_p;p^{k+1}) \supseteq \SL_n(\mathbb{Z}_p;p^k)$, we get that $X^{5C} \supseteq \SL_n(\mathbb{Z}_p;p^k)$, proving the claim in this case.

\end{itemize} 
\item Since $w(\SL_n(\widehat{\mathbb{Z}}))=\prod w(\SL_n(\mathbb{Z}_p))$, the claim follows from the first claim.
\end{enumerate} 
\end{proof} 

\begin{proof}[Proof of Theorem \ref{thm:intro.fix.n}] By Theorem \ref{thm:intro.congruence.subgroup} and Lemma \ref{lem:width.adelic.fix.n}.
\end{proof}

\section{Proof of Theorems \ref{thm:intro.fix.w} and  \ref{thm:intro.congruence.subgroup}  - with explicit bounds}

The goal of this section is to prove the explicit bound of Theorem \ref{thm:intro.fix.w}. The proof follows the arguments in \cite{DV}.

\begin{lemma}\label{lemma:explicit.1} Let $q,m \in \N^+$ and denote $n:=3m$. Assume that $g_1,\ldots,g_m \in \SL_3(\Z;q)$ and $g_1\cdots g_m=e$. Then 
$g:=\diag(g_1,\ldots,g_m) \in L_n(\Z;q)\tilde{U}_n(\Z;q) U_n(\Z;q)$ where   $\tilde{U}_n(\Z;q):=\{hkh^{-1} \mid k \in U_n(\Z;q) \ \wedge \ h \in \SL_n(\Z)\}$.
\end{lemma}
\begin{proof} Let $I_3$ be the identity matrix of $\SL_3(\Z)$ and identify $M_n(\Z)$ with $M_m(M_3(\Z))$ where 
$M_k(R)$ is the ring of $k \times k $ matrices over the ring $R$. 
Let $l_1$ be the matrix of $M_m(M_3(\Z))$ with $I_3$
on the diagonal, $g_i^{-1}$ on the $(i+1,i)$ entry for every $1 \le i \le m-1$, and zero elsewhere. 
Let $l_2$ be the matrix of $M_m(M_3(\Z))$ with $I_3$
on the diagonal, $I_3$ on the $(i+1,i)$ entry for every $1 \le i \le m-1$, and zero elsewhere. 
Let $u_1$ be the matrix of $M_m(M_3(\Z))$ with $I_3$
on the diagonal, $1-g_1\cdots g_i$ on the $(i,i+1)$ entry for every $1 \le i \le m-1$, and zero elsewhere. 
Let $u_2$ be the matrix of $M_m(M_3(\Z))$ with $I_3$
on the diagonal, $(1-g_1 \cdots g_i)g_{i+1}$ on the $(i,i+1)$ entry for every $1 \le i \le m-1$, and zero elsewhere. 
Then, $g=l_1^{-1}u_1^{-1}l_2u_2=(l_1^{-1}l_2)(l_2^{-1}u_1^{-1}l_2)u_2 \in L_n(\Z;q)\tilde{U}_n(\Z;q) U_n(\Z;q)$.
\end{proof}

\begin{lemma}\label{lemma:explicit.2} Let $q,m \in \N^+$ and denote $n:=3m$. Assume that $g_1,\ldots,g_m \in U_3(\Z;q)L_3(\Z;q)$. Denote
$g:=\diag(g_1 \cdots g_m,I_3,\ldots,I_3) \in \SL_n(\Z;q)$ and $\tilde{U}_n(\Z;q):=\{hkh^{-1} \mid k \in U_n(\Z;q) \ \wedge \ h \in \SL_n(\Z)\}$.
Then $g \in L_n(\Z;q)\tilde{U}_n(\Z;q) U_n(\Z;q)L_n(\Z;q)$.
\end{lemma}
\begin{proof} Define $h:=\diag(g_m,g_{m-1},\ldots, g_1)\in U_n(\Z;q)L_n(\Z;q)$. Lemma \ref{lemma:explicit.1} implies that 
$gh^{-1} \in L_n(\Z;q)\tilde{U}_n(\Z;q) U_n(\Z;q)$. Thus, $$g \in L_n(\Z;q)\tilde{U}_n(\Z;q) U_n(\Z;q)h \subseteq L_n(\Z;q)\tilde{U}_n(\Z;q) U_n(\Z;q) L_n(\Z;q).$$
\end{proof}

\begin{lemma}\label{lemma:explicit.3}
 Let $n \ge m \ge  3$ and $q \ge 1$.  Denote ${E}^*(m,\Z;q):=\{\diag(1,\ldots,1,g) \in \SL_n(\Z) \mid g \in {E}(m,\Z;q)\}$.   Then, ${E}(n,\Z;q) = L_n(\Z;q)U_n(\Z;q)L_n(\Z;q){E}^*(m,\Z;q)U_n(\Z;q)$.
\end{lemma}

\begin{proof} 
Let $q \ge 1$. The proof is by induction on $n$. The base case $n=m$ is clear. Assume that the statement  is true for some $n \ge m$. We have to show that 
the statement is true also for $n+1$.  Let $U_n^-(\Z;q)$ and $L_n^-(\Z;q)$ be the images in $\SL_{n+1}(\Z)$ of $U_n(\Z;q)$ and $L_n(\Z;q)$ under the map $M \mapsto \diag(1,M)$. Denote $C^-_n(q):=\la e_{j,1}(q) \mid 2 \le j \le n+1\ra$ and 
$R^-_n(q):=\la e_{1,j}(q) \mid 2 \le j \le n+1\ra$. Finally recall that the main theorem of \cite{Men} implies that for every $k \ge 3$,
$$E(k,\Z;q)=\{g\in \SL_k(\Z;q) \mid \forall 1 \le i \le k. \ g_{i,i}=1 (\textrm{mod}\ q^2)  \}.$$
Let $g \in{E}(n+1,\Z;q)$. Then,  $\gcd(g_{1,1},g_{2,1},\ldots,g_{n,1})=1$ and $\gcd(qg_{1,1},g_{2,1},\ldots,g_{n,1})=q$. 
Recall that $\Z$ satisfies 
the following stable range condition: If $m\ge 3$ and $a_1,\ldots,a_m \in \Z$  then there exists $t_2,\ldots,t_m \in \Z$ such that 
$\gcd(a_1,\ldots,a_m)=\gcd(a_2-t_2a_1,\ldots,a_n-t_na_1)$. Thus, there exists $h\in C^-_n(\Z;q)g$ such that 
$\gcd(h_{2,1},\ldots,h_{n,1})=q$. Since $h \in {E}(n,\Z;q)$, we have $h_{1,1}=1$ modulo $q^2$ so there exists $h'\in R^-_n(\Z;q)h$ such that $h'_{1,1}=1$. Finally, there is $h''\in C^-_n(q)h'R^-_n(q)$ such that $h''=\diag(1,g')$ for some $g' \in \SL_n(\Z;q)$.  Note that $g' \in {E}(n,\Z;q)$ since its diagonal entries are equal to 1 modulo $q^2$. 
Thus, the induction hypothesis implies that $h'' \in L_n^-(\Z;q)U_n^-(\Z;q)L_n^-(\Z;q)E^*(m,\Z;q)U_n^-(\Z;q)$. 
It follows that $g$ belongs to $$C^-_n(\mathbb{Z};q)R^-_n(\mathbb{Z};q)C^-_n(\mathbb{Z};q)L_n^-(\Z;q)U_n^-(\Z;q)L_n^-(\Z;q)E^*(m,\Z;q)U_n^-(\Z;q)R^-_n(\mathbb{Z};q).$$
Since both $U_n^-(\Z;q)$ and $L_n^-(\Z;q)$ normalize $C^-_n(\mathbb{Z};q)$ and $R^-_n(\mathbb{Z};q)$, we have

\[
C^-_n(\mathbb{Z};q)R^-_n(\mathbb{Z};q)C^-_n(\mathbb{Z};q)L_n^-(\Z;q)U_n^-(\Z;q)L_n^-(\Z;q)E^*(m,\Z;q)U_n^-(\Z;q)R^-_n(\mathbb{Z};q)=
\]
\[
C^-_n(\mathbb{Z};q)L_n^-(\Z;q)R^-_n(\mathbb{Z};q)U_n^-(\Z;q)C^-_n(\mathbb{Z};q)L_n^-(\Z;q)E^*(m,\Z;q)U_n^-(\Z;q)R^-_n(\mathbb{Z};q) =
\]
\[
L_{n+1}(\Z;q)U_{n+1}(\Z;q)L_{n+1}(\Z;q)E^*(m,\Z;q)U_{n+1}(\Z;q)
\]
\end{proof}

\begin{corollary}\label{cor:explicit} For every $n \ge 3$ denote $\tilde{U}_n(\Z;q):=\{hkh^{-1} \mid k \in U_n(\Z;q) \ \wedge \ h \in \SL_n(\Z)\}=\{hkh^{-1} \mid k \in L_n(\Z;q) \ \wedge \ h \in \SL_n(\Z)\}$.
There exists an integer $N$ such that for every $n \ge N$ and every $q \in \Z$ the following holds:
$$
E(n,\Z;q) \subseteq    L_n(\Z;q)(\tilde{U}_n(\Z;q))^3 U_n(\Z;q).
$$
\end{corollary}

\begin{proof}
Propositon \ref{claim:bdd.gen.3} implies that there exists a constant $D$ such that for every $q \in \Z$,
$(U_3(\Z;q)L_3(\Z;q))^D=\la U_3(\Z;q)L_3(\Z;q) \ra$.
Denote ${E}^*(3,\Z;q):=\{\diag(1,\ldots,1,g) \in \SL_{3D}(\Z) \mid g \in {E}(3,\Z;q)\}$.   
Lemma \ref{lemma:explicit.2}  shows that, for any $q$, $E^*(3,\mathbb{Z};q) \subseteq L_{3D}(\mathbb{Z};q)\widetilde{U}_{3D}(\mathbb{Z};q)U_{3D}(\mathbb{Z};q)L_{3D}(\mathbb{Z};q)$. 
Lemma \ref{lemma:explicit.3} implies that

$$E(n,\Z;q) \subseteq   L_{n}(\Z;q)U_{n}(\Z;q)L_{n}(\Z;q)\tilde{U}_{n}(\Z;q) U_{n}(\Z;q)L_{n}(\Z;q)U_{n}(\Z;q).$$

Since $L_{n}(\Z;q)U_{n}(\Z;q)L_{n}(\Z;q) \subset L_n(\mathbb{Z};q)\tilde{U}_n(\mathbb{Z};q)$ and $U_{n}(\Z;q)L_{n}(\Z;q)U_{n}(\Z;q) \subset \tilde{U}_n(\mathbb{Z};q)U_n(\mathbb{Z};q)$, we get the result.
\end{proof}

\begin{proof}[Proof of Theorems \ref{thm:intro.fix.w} and \ref{thm:intro.congruence.subgroup} (with explicit bounds)] Let $n \ge 3$. 
The proof of Theorem \ref{thm:intro.congruence.subgroup} shows that there is a non-zero $q \in \Z$ such that 
$U_n(\Z;q)$ and $L_n(\Z;q)$ are contained in $w(\SL_n(\Z))^{16}$. Since $w(\SL_n(\Z))$ is a normal subset, the set 
$\tilde{U}_n(\Z;q):=\{hkh^{-1} \mid k \in U_n(\Z;q) \ \wedge \ h \in \SL_n(\Z)\}$ is contained in $w(\SL_n(\Z))^{16}$. Corollary 
\ref{cor:explicit} implies that if $n$ is large enough then $E(n,\Z;q) \subseteq w(\SL_n(\Z))^{16\cdot 5}$. Since $E(n,\mathbb{Z};q)$ contains a congruence subgroup, we proved the bound in Theorem \ref{thm:intro.congruence.subgroup}. Finally, Lemma  \ref{lem:width.adelic.fix.w} implies that if $n$ is large enough then $\SL_n(\Z) \subseteq w(\SL_n(\Z))^{16\cdot 5+7}$.
\end{proof}

\end{document}